\documentclass[11pt]{article}
\usepackage[margin =1in]{geometry}
\usepackage{amssymb,amsthm,amsmath}
\usepackage{enumerate}
\usepackage{hyperref}
\usepackage{cite}
\usepackage[capitalize]{cleveref}
\usepackage{enumitem}
\usepackage{color}
\usepackage{cancel}
\usepackage{svg}
\usepackage{import}
\usepackage{optidef}

\usepackage[utf8]{inputenc} % allow utf-8 input
\usepackage[T1]{fontenc}    % use 8-bit T1 fonts
\usepackage{hyperref}       % hyperlinks
\usepackage{url}            % simple URL typesetting
\usepackage{booktabs}       % professional-quality tables
\usepackage{amsfonts}       % blackboard math symbols
\usepackage{nicefrac}       % compact symbols for 1/2, etc.
\usepackage{microtype}
\usepackage{multirow}
\usepackage{xfrac}
\usepackage{todonotes}
\usepackage{titlesec}
\usepackage{caption}
\usepackage{subcaption}
\usepackage{amsmath}

\titleformat{\subsubsection}
{\normalfont\normalsize\bfseries}{\thesubsubsection}{1em}{}
\usepackage{graphicx}

\usepackage{appendix}
\numberwithin{equation}{section}
\newcommand{\tr}{\mathrm{Tr}}

\newcommand{\inclu}[0] {\ar@{^{(}->}}

\newtheorem{thm}{Theorem}[section]
\newtheorem{theorem}{Theorem}[section]
\newtheorem{conjecture}{Conjecture}[section]

\newtheorem{lemma}[thm]{Lemma}

\crefname{claim}{claim}{claims}
\Crefname{claim}{Claim}{Claims}
\crefname{lem}{lemma}{lemmas}
\Crefname{lem}{Lemma}{Lemmas}
\crefname{algorithm}{algorithm}{algorithms}
\Crefname{algorithm}{Algorithm}{Algorithms}

\newtheorem{example}{Example}[section]

\theoremstyle{remark}

\newcommand{\set}[1]{\left\{ #1 \right\}}

\usepackage{xcolor}

\definecolor{blue}{RGB}{0,0,0}

\usepackage{mathtools}

\usepackage[algo2e, boxruled]{algorithm2e}

\usepackage[T1]{fontenc}
\usepackage{lmodern}

\begin{document}

\title{Provably Faster Gradient Descent via Long Steps}

	 \author{Benjamin Grimmer\footnote{Johns Hopkins University, Department of Applied Mathematics and Statistics, \texttt{grimmer@jhu.edu}}}

	\date{}
	\maketitle

	\begin{abstract}
       This work establishes new convergence guarantees for gradient descent in smooth convex optimization via a computer-assisted analysis technique. Our theory allows nonconstant stepsize policies with frequent long steps potentially violating descent by analyzing the overall effect of many iterations at once rather than the typical one-iteration inductions used in most first-order method analyses. We show that long steps, which may increase the objective value in the short term, lead to provably faster convergence in the long term. A conjecture towards proving a faster $O(1/T\log T)$ rate for gradient descent is also motivated along with simple numerical validation.
	\end{abstract} 

    \section{Introduction}
{\color{blue} This work proposes a new analysis technique for gradient descent, establishing a path toward provably better convergence guarantees for smooth, convex optimization than is possible with existing approaches based on constant stepsizes. Instead,} our theory allows for nonconstant stepsize policies, periodically taking larger steps that may violate the monotone decrease in objective value typically needed by analysis. In fact, contrary to the common intuition, we show periodic long steps, which may increase the objective value in the short term, provably speed up convergence in the long term, with increasingly large gains as longer and longer steps are periodically included. 
This bears a similarity to accelerated momentum methods, which also depart from ensuring a monotone objective decrease at every iteration.

Establishing this requires a proof technique capable of analyzing the overall effect of many iterations at once rather than the typical (naive) one-iteration inductions used in most first-order method analyses. Our proofs are based on the Performance Estimation Problem (PEP) ideas of~\cite{drori2012PerformanceOF,taylor2017interpolation,taylor2017smooth}, which cast computing/bounding the worst-case problem instance of a given algorithm as a Semidefinite Program (SDP). We show that the existence of a feasible solution to a related SDP proves a descent guarantee after applying a corresponding pattern of nonconstant stepsizes, from which faster convergence guarantees follow. Our technique is very similar to that first proposed by Altschuler's Master's thesis~\cite[Chapter 8]{altschuler2018greed} which established repeating stepsize patterns of length two or three with faster contractions towards the minimizer for smooth, strongly convex minimization.
Most existing PEP literature uses computer-solves to guide the search for tighter convergence proofs~\cite{Klerk2016OnTW,taylor2017smooth,taylor2018Lyapunov,Dragomir2019OptimalCA,Taylor2019StochasticFM,Lieder2020OnTC,Ryu2020tightSplittingContraction,Gu2020tightPPM,Barre2020PrincipledAA,gupta2023branch} and inform the development of new, provably faster algorithms~\cite{Kim2016optimal,Drori2018EfficientFM,Kim2021gradient,Kim2021AccelPPM,taylor2022optimal}. Here computer outputs are directly used to constitute the proof, but are large (up to tens of thousands of rational numbers) and so may provide less guidance or intuition.

We consider gradient descent with a sequence of (normalized) stepsizes $h = (h_0,h_1,h_2,\dots)$ applied to minimize a convex function $f \colon \mathbb{R}^n\rightarrow \mathbb{R}$ with $L$-Lipschitz gradient by iterating
\begin{equation}
    x_{k+1} = x_k - \frac{h_k}{L} \nabla f(x_k) 
\end{equation}
given an initialization $x_0\in\mathbb{R}^n$. We assume throughout that a minimizer $x_\star$ of $f$ exists and its level sets are bounded $D=\sup\{\|x-x_\star\|_2 \mid f(x)\leq f(x_0)\}$\footnote{This assumption can likely be relaxed but eases our development herein.}. 
The classic convergence guarantee~\cite{Bertsekas2015ConvexOA} for gradient descent is that with constant stepsizes $h=(1, 1, 1, \dots)$, every $T>0$ has $ f(x_T) - f(x_\star) \leq LD^2/2(T+1)$.
{\color{blue} Using performance estimation techniques, a tight refinement of this bound was given by~\cite[Theorem 3.1]{drori2012PerformanceOF} (and elementary proof, avoiding the use of PEP, was given by~\cite{Teboulle2023}): every $T>0$ has
\begin{equation} \label{eq:distance-induction-style-result}
    f(x_T) - f(x_\star) \leq \frac{LD^2}{4T+2} \ .
\end{equation}
By tight, we mean a matching problem instance exists that attains the above inequality. Using stepsizes between one and two, this rate can be improved by another factor of two (see~\cite{taylor2017interpolation,Teboulle2023}). Stepsizes beyond length two have very little prior theory as one can no longer guarantee a decrease in objective value at each iteration.
}

{\color{blue} Here we provide an analysis technique capable of handling nonconstant stepsizes, periodically longer than one can guarantee descent for, finding increasingly fast convergence (in terms of constants) follows.} For example, consider gradient descent alternating stepsizes $h = (2.9,1.5,\ 2.9,1.5,\  \dots)$. Such a scheme is beyond the reach of traditional descent-based analysis as one may fear the stepsizes of $2.9$ can increase the function value individually more than the $1.5$ is guaranteed to decrease it. Regardless, {\color{blue} we show this ``long step'' converges with}
$$ f(x_T)-f(x_\star) \leq \frac{LD^2}{2.2 \times T} + O(1/T^2) $$
for every even $T>0$. See our Theorem~\ref{thm:two-step-rate} characterizing many such alternating stepsize methods. The $+O(1/T^2)$ terms throughout this work are only used to suppress two universal constants, namely above, we show there exist constants $\bar s$ and $C$ such that all even $T > 2\bar s$ have bound $LD^2/(2.2 \times T - C)$.

Using longer cycles, we derive further performance gains. For example, we show a carefully selected stepsize pattern of length $127$ periodically taking stepsizes of $370.0$ converges at a rate of $LD^2/(5.8346303\times T)$. Generally, given a stepsize pattern $h=(h_0,\dots,h_{t-1})\in\mathbb{R}^t$, we consider the gradient descent method repeatedly applying the pattern of stepsizes 
\begin{equation}\label{eq:pattern-GD}
    x_{k+1} = x_k - \frac{h_{(k \mathtt{\ mod\ } t)}}{L} \nabla f(x_k) \ .
\end{equation}
In Theorem~\ref{thm:straightforward-rates}, we give a convergence guarantee for any \emph{straightforward} stepsize pattern $h$ of
\begin{equation} \label{eq:pattern-GD-rate}
    f(x_{T}) - f(x_\star) \leq \frac{LD^2}{\mathrm{avg}(h) T} + O(1/T^2) 
\end{equation}
{\color{blue} where $\mathrm{avg}(h) = \frac{1}{t}\sum_{i=0}^{t-1}h_i$.}
See Section~\ref{sec:straightforward} for the formal introduction of this straightforwardness property.
Hence the design of provably faster nonconstant stepsize gradient descent methods amounts to seeking straightforward stepsize patterns with large average stepsize values. Certifying a given pattern is straightforward can be done via semidefinite programming (see our Theorem~\ref{thm:straightforward-certificate}). So the convergence analysis of such nonconstant stepsizes methods is a natural candidate for computer assistance. 
\begin{table}[p]\centering
\begin{tabular}{|c|p{60mm}|p{35mm}|}\hline
     Pattern Length & ``Straightforward'' Stepsize Pattern $h$ \newline (longest stepsize marked in bold) &  Convergence Rate \newline ($+O(1/T^2)$ omitted)\\ \hline
     $t=2$ & $({\bf 3-\eta}, 1.5) \qquad \text{for\ any\ } \eta\in (0,3)$ &  $\dfrac{LD^2}{(2.25-\eta/2) \times T}$\\     
     $t=3$ & $(1.5, {\bf 4.9}, 1.5)$ & $\dfrac{LD^2}{2.63333...\times T}$ \\   
%     $t=4$ & $(1.9, 1.4, {\bf 6.0}, 1.5)$ & \hfill$\dfrac{LD^2}{\dots\times T} +O\left(\dfrac{1}{T^2}\right)$ \\     
     $t=7$ & $(1.5, 2.2, 1.5, {\bf 12.0}, 1.5, 2.2, 1.5)$ & $\dfrac{LD^2}{3.1999999\times T}$ \\
     $t=15$ &
     $(1.4, 2.0, 1.4, 4.5, 1.4, 2.0, 1.4, {\bf 29.7}, $\newline
     $ \ ~\ 1.4, 2.0, 1.4, 4.5, 1.4, 2.0, 1.4)$ & $\dfrac{LD^2}{3.8599999\times T}$ \\
     $t=31$ &
     $(1.4, 2.0, 1.4, 3.9, 1.4, 2.0, 1.4, 8.2, $\newline
     $ \ ~\ 1.4, 2.0, 1.4, 3.9, 1.4, 2.0, 1.4, {\bf 72.3},$\newline
     $ \ ~\ 1.4, 2.0, 1.4, 3.9, 1.4, 2.0, 1.4, 8.2, $\newline
     $ \ ~\ 1.4, 2.0, 1.4, 3.9, 1.4, 2.0, 1.4) $
     & $\dfrac{LD^2}{4.6032258\times T}$\\
     $t=63$ &
     $(1.4, 2.0, 1.4, 3.9, 1.4, 2.0, 1.4, 7.2, $\newline
     $ \ ~\ 1.4, 2.0, 1.4, 3.9, 1.4, 2.0, 1.4, 14.2,$\newline
     $ \ ~\ 1.4, 2.0, 1.4, 3.9, 1.4, 2.0, 1.4, 7.2, $\newline
     $ \ ~\ 1.4, 2.0, 1.4, 3.9, 1.4, 2.0, 1.4, {\bf 164.0},$ \newline
     $ \ ~\ 1.4, 2.0, 1.4, 3.9, 1.4, 2.0, 1.4, 7.2, $\newline
     $ \ ~\ 1.4, 2.0, 1.4, 3.9, 1.4, 2.0, 1.4, 14.2,$\newline
     $ \ ~\ 1.4, 2.0, 1.4, 3.9, 1.4, 2.0, 1.4, 7.2, $\newline
     $ \ ~\ 1.4, 2.0, 1.4, 3.9, 1.4, 2.0, 1.4)$
     & $\dfrac{LD^2}{5.2253968\times T}$\\
     $t=127$ & $(1.4, 2.0, 1.4, 3.9, 1.4, 2.0, 1.4, 7.2, $\newline
     $ \ ~\ 1.4, 2.0, 1.4, 3.9, 1.4, 2.0, 1.4, 12.6,$\newline
     $ \ ~\ 1.4, 2.0, 1.4, 3.9, 1.4, 2.0, 1.4, 7.2, $\newline
     $ \ ~\ 1.4, 2.0, 1.4, 3.9, 1.4, 2.0, 1.4, 23.5,$ \newline
     $ \ ~\ 1.4, 2.0, 1.4, 3.9, 1.4, 2.0, 1.4, 7.2, $\newline
     $ \ ~\ 1.4, 2.0, 1.4, 3.9, 1.4, 2.0, 1.4, 12.6,$\newline
     $ \ ~\ 1.4, 2.0, 1.4, 3.9, 1.4, 2.0, 1.4, 7.2, $\newline
     $ \ ~\ 1.4, 2.0, 1.4, 3.9, 1.4, 2.0, 1.4, {\bf 370.0}, $\newline
     $ \ ~\ 1.4, 2.0, 1.4, 3.9, 1.4, 2.0, 1.4, 7.2, $\newline
     $ \ ~\ 1.4, 2.0, 1.4, 3.9, 1.4, 2.0, 1.4, 12.6,$\newline
     $ \ ~\ 1.4, 2.0, 1.4, 3.9, 1.4, 2.0, 1.4, 7.2, $\newline
     $ \ ~\ 1.4, 2.0, 1.4, 3.9, 1.4, 2.0, 1.4, 23.5,$ \newline
     $ \ ~\ 1.4, 2.0, 1.4, 3.9, 1.4, 2.0, 1.4, 7.2, $\newline
     $ \ ~\ 1.4, 2.0, 1.4, 3.9, 1.4, 2.0, 1.4, 12.6,$\newline
     $ \ ~\ 1.4, 2.0, 1.4, 3.9, 1.4, 2.0, 1.4, 7.2, $\newline
     $ \ ~\ 1.4, 2.0, 1.4, 3.9, 1.4, 2.0, 1.4)$ & $\dfrac{LD^2}{5.8346303\times T}$\\ \hline
\end{tabular}
\caption{Improved convergence guarantees for Gradient Descent with stepsizes cycling through a ``straightforward'' pattern. Each convergence bound is proven by producing a certificate of feasibility for a related SDP, which is sufficient by our Theorems~\ref{thm:straightforward-rates} and~\ref{thm:straightforward-certificate}. Coefficients for $t\geq 7$ guarantees are slightly smaller than the ideal $\mathrm{avg}(h)$ due to rounding to produce an exact arithmetic certificate.}\label{tab:patterns-and-rates}
\end{table}

Table~\ref{tab:patterns-and-rates} shows straightforward stepsize patterns with increasingly fast convergence guarantees, each proven using a computer-generated, exact-arithmetic semidefinite programming solution certificate. 
Future works identifying longer straightforward patterns and other tractable families of nonconstant, periodically long stepsize policies will surely be able to improve on this work's particular guarantees.

The analysis of such nonconstant, long stepsize gradient descent methods has eluded the literature, with only a few exceptions. In 1953, Young~\cite{young1953} showed optimal, accelerated convergence is possible for gradient descent when minimizing a smooth, strongly convex quadratic function by using a careful nonconstant selection of $h_i$. Namely, Young set $h_0\dots h_{T-1}$ as one over the roots of the $T$-degree Chebyshev polynomial\footnote{A nice summary of this is given by the recent blog post~\cite{pedregosa2021nomomentum}.}.
Few works have shown faster convergence from long steps beyond quadratics. Two notable works have done so for smooth, strongly convex minimization:

Oymak~\cite{Oymak2021} showed substantial speed-ups for strongly convex functions with special bimodal structured Hessians.
{\color{blue} Closely in line with this work's reasoning, several recent Master's and doctoral theses have addressed the optimal design and analysis of one, two, or three steps of  gradient descent. The theses of Daccache~\cite{Daccache2019} and Eloi~\cite{Eloi2022} provide exhaustive characterizations of gradient descent's objective gap after two or three steps for smooth convex optimization. Their results do not provide a mechanism to be applied inductively, so no convergence rates follow from repeatedly applying the studied patterns.} For strongly convex optimization, Altschuler~\cite[Chapter 8]{altschuler2018greed} used PEP techniques to derive the optimal patterns of length $t=1,2,3$ for contracting either the distance to optimal or the gradient's norm. Inductively applying their improved contractions, one arrives at the same moral takeaway advanced by this work: SDP-based analysis can prove faster convergence follows from periodically taking longer steps, potentially violating descent. Importantly, their length two and three patterns are optimal, meaning their patterns give the best possible contraction factor. Note their patterns differ from those presented here, only being optimal for strongly convex problems.
{\it The primary contribution in this work is identifying a tractable analysis technique for general smooth, convex optimization and showing increasing performance gains for increasingly large $t>3$, continuing in this seventy-year-old direction.}

The search for long, straightforward stepsize patterns $h$ is hard; the set of all straightforward patterns is nonconvex, making local searches often unfruitful.
Our patterns of length $t=2^{m}-1$ in Table~\ref{tab:patterns-and-rates} were created by repeating the pattern for $t=2^{m-1}-1$ twice with a new long step added in between and (by hand) shrinking the long steps in the length $2^{m-1}-1$ subpatterns. This recursive pattern has strong similarities to the cyclic and fractal Chebyshev patterns for quadratic minimization considered by~\cite{LEBEDEV1971155,Agarwal2021AccelerationVF,Goujaud2022}, although we make no provable connections. This doubling procedure consistently increased $\mathrm{avg}(h)$ by $\approx 0.6$. We conjecture the following.
\begin{conjecture}
    For any $t\in\mathbb{N}$, there exists a straightforward stepsize pattern $h\in\mathbb{R}^t$ with $\mathrm{avg}(h) = \Omega(\log(t))$.
\end{conjecture}
\noindent If true, this would likely yield convergence rates on the order of $O(1/(T\log(T)))$, strictly improving on the classic $O(1/T)$ guarantee. If such long patterns exist, one natural question is how close to the optimal $O(1/T^2)$ rate attained by momentum methods can be achieved by gradient descent with long steps. The numerics of~\cite[Figure 2]{gupta2023branch} suggested a $O(1/T^{1.178})$ rate may be possible. The theory of Lee and Wright~\cite{Lee2019} showed asymptotic $o(1/T)$ convergence for constant stepsize gradient descent (although no improved guarantees in finite-time were given), which also motivates the possibility for improved finite-time convergence guarantees.

\begin{figure}
\includegraphics[width=0.5\linewidth]{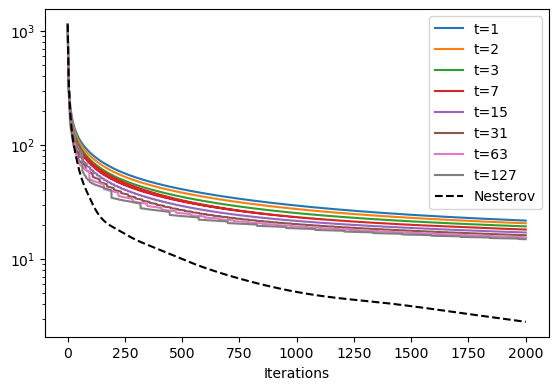}\includegraphics[width=0.5\linewidth]{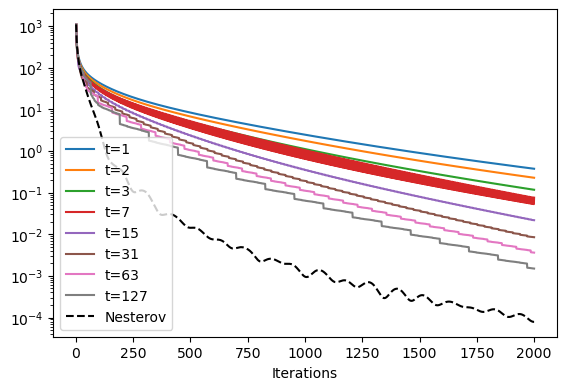}
\caption{\color{blue} Least squares problems minimizing $\|Ax-b\|^2_2$ (left) and $\|Ax-b\|^2_2 + \|x\|^2_2$ (right) with i.i.d.~normal entries in $A\in\mathbb{R}^{n\times n}$ and $b\in\mathbb{R}^n$ for $n=4000$. Objective gaps are plotted over $T=2000$ iterations with $h=(1)$ and with each pattern from Table~\ref{tab:patterns-and-rates}. Note this second objective is substantially more strongly convex, so its faster linear convergence is expected. For comparison, Nesterov's accelerated method (without modification to utilize strong convexity) is shown.}\label{fig:numeric}
\end{figure}

For strongly convex optimization (or, more generally, any problem satisfying a H\"older growth bound), the classic convergence rates for constant stepsize gradient descent are known to improve. We show the same improvements occur for any straightforward stepsize pattern (see Theorem~\ref{thm:straightforward-rates-SC}) with an additional gain of $\mathrm{avg}(h)$.
We validate that the convergence speed-ups of Table~\ref{tab:patterns-and-rates} actually occur on randomly generated least squares problems in Figure~\ref{fig:numeric}, seeing gains proportional to $\mathrm{avg}(h)$. These plots also showcase that descent is not ensured within the execution of a straightforward pattern as $t=7$ rapidly oscillates within each pattern while converging overall. {\color{blue} Note, gradient descent is not an optimal first-order method as accelerated methods can attain a faster $O(1/T^2)$ rate. A basic implementation of Nesterov's accelerated method is included in our numerics to showcase the degree of progress in closing this gap able to be accomplished via our longer steps.}

{\color{blue} Cyclic, periodically long stepsizes have been used in neural network training schedules~\cite{Loshchilov2016SGDRSG,Smith2015CyclicalLR,Smith2017SuperconvergenceVF}. Our results only apply to convex, deterministic problems. New techniques handling nonconvexity and stochasticity would need to be developed to describe the effect of long steps in such settings of machine learning.} \\[-8pt]

\noindent {\bf Outline.} In the remainder of this section, we informally sketch how our proof technique proceeds. Then Section~\ref{sec:straightforward} formally introduces our notion of straightforward stepsize patterns,
showing that any such pattern has a guarantee of the form~\eqref{eq:pattern-GD-rate}. Section~\ref{sec:certificate} shows the existence
of a solution to a certain semidefinite program implies straightforwardness. Section~\ref{sec:conclusion} concludes by outlining several future directions of interest enabled by and hopefully able to improve on this work.

\subsection{Sketch of Proof Technique - Reducing proving eventual descent to an SDP}
Our analysis works by guaranteeing a sufficient decrease is achieved after applying the whole pattern $h=(h_0,h_1,\dots,h_{t-1})$ of $t$ steps (but not necessarily descending at any of the intermediate iterates). Our notion of straightforward stepsize patterns aims to ensure that for all $\delta>0$ small enough, if $f(x_0)-f(x_\star) \leq \delta$, then $x_{t}$ will always attain a descent of at least
\begin{equation} \label{eq:descent-target}
    f(x_t) -f(x_\star) \leq \delta - \frac{\sum_{i=0}^{t-1} h_i}{LD^2} \delta^2 \ .
\end{equation}
For constant stepsizes equal to one, this amounts to $f(x_0-\nabla f(x_0)/L)\leq f(x_0) - \delta^2/LD^2$, a classic descent result that holds for all $L$-smooth convex $f$.

{\color{blue} To prove a descent lemma like~\eqref{eq:descent-target}, one can take a direct combination of several known inequalities. This is a well-known approach, equivalent to providing dual solutions to the dual of the associated performance estimation problem.} We are given the equalities $x_{k+1}=x_k-(h_k/L)\nabla f(x_k)$ and $\nabla f(x_\star)=0$ and as inequalities an initial distance bound
$$ \|x_0-x_\star\|^2_2 \leq D^2 \ ,$$
an initial objective gap bound
$$ f(x_0) - f(x_\star) \leq \delta \ , $$
and for any $x_i$ and $x_j$ with $i,j\in\{\star,0,1,2,\dots t\}$, convexity and smoothness imply\cite[(2.1.10)]{nesterov-textbook}
$$ f(x_i) \geq f(x_j) + \nabla f(x_j)^T(x_i-x_j) + \frac{1}{2L}\|\nabla f(x_i) - \nabla f(x_j)\|^2_2 \ . $$
For any nonnegative multipliers $v(\delta),w(\delta),\lambda_{i,j}(\delta)\geq 0$ (parameterized by the initial objective gap $\delta$), one can combine these inequalities to conclude that on every $L$-smooth convex function, gradient descent~\eqref{eq:pattern-GD} with stepsizes $h$ satisfies
\begin{align}
     v(\delta)&(\|x_0-x_\star\|^2_2 - D^2)\label{eq:total-inequality}\\
    + w(\delta)&(f(x_0) - f(x_\star) - \delta)\nonumber\\
     + \sum_{i,j\in\{\star,0,\dots t\}} \lambda_{i,j}(\delta)&\left(f(x_j) + \nabla f(x_j)^T(x_i-x_j) + \frac{1}{2L}\|\nabla f(x_i) - \nabla f(x_j)\|^2_2 - f(x_i)\right)\leq 0 \ .\nonumber
\end{align}
Our proof then proceeds by showing carefully selected functions $w(\delta),v(\delta),\lambda(\delta)$ reduce this inequality to guaranteeing~\eqref{eq:descent-target}. We find to prove our guarantees in Table~\ref{tab:patterns-and-rates}, it suffices to set $v(\delta)=\frac{\sum_{i=0}^{t-1} h_i}{LD^4}\delta^2$, $w(\delta) = 1 - \frac{2\sum_{i=0}^{t-1} h_i}{LD^2} \delta$, and use a linear function $\lambda(\delta) = \lambda + \delta \gamma$. For this choice of $v(\delta)$ and $w(\delta)$, our Theorem~\ref{thm:straightforward-certificate} shows~\eqref{eq:total-inequality} implies~\eqref{eq:descent-target} if $(\lambda,\gamma)$ is feasible to a certain SDP. Hence, proving a $t$ iteration descent lemma of the form~\eqref{eq:descent-target} can be done by semidefinite programming. Our Theorem~\ref{thm:straightforward-rates} then completes the argument by showing such a periodic descent guarantee implies the guarantee~\eqref{eq:pattern-GD-rate}.
    \section{Straightforward Stepsize Patterns}\label{sec:straightforward}
To analyze a given stepsize pattern, we first aim to understand its worst-case problem instance. We do so through the ``Performance-Estimation Problem'' (PEP) framework of~\cite{drori2012PerformanceOF}. Given a pattern $h$ and bounds on smoothness $L$, initial distance to optimal $D$, and initial objective gap $\delta$, the worst final objective gap able to be produced by one application of the stepsize pattern is given by
\begin{equation} \label{eq:value-function}
    p_{L,D}(\delta) := \begin{cases}
    \max_{x_0,x_\star,f} & f(x_t) - f(x_\star) \\
    \mathrm{s.t.} & f \text{ is convex, } L\text{-smooth}\\
    & \|x_0-x_\star\|_2 \leq D\\
    & f(x_0) - f(x_\star) \leq \delta \\
    & \nabla f(x_\star)=0 \\
    & x_{k+1} = x_k - \frac{h_k}{L} \nabla f(x_k) \quad \forall k=0,\dots, t-1 \ .
    \end{cases}
\end{equation}
{\color{blue} Note this problem is infinite-dimensional as it optimizes over functions $f\colon \mathbb{R}^n\rightarrow \mathbb{R}$ and vectors $x_0,x_\star\in\mathbb{R}^n$ for some given problem dimension $n$.}
The key insight of Drori and Teboulle~\cite{drori2012PerformanceOF} is that such infinite-dimensional problems can be relaxed to a finite-dimensional semidefinite program. Subsequently, Taylor et al.~\cite{taylor2017interpolation} showed this relaxation can be made an exact reformulation. We formalize and utilize these insights in Section~\ref{sec:certificate}. {\color{blue} We note one particular nice property here: $p_{L,D}$ is concave.
\begin{lemma}\label{lem:p-concave}
    For any $L,D>0$, $p_{L,D}$ is concave.
\end{lemma}
\begin{proof}
    Consider any $\delta^{(1)}, \delta^{(2)}\in\mathbb{R}$ and $\lambda\in(0,1)$. Denote feasible solutions to the problem defining $p_{L,D}(\delta^{(i)})$ for each $i=1,2$ by $f^{(i)}, x^{(i)}_0,x^{(i)}_\star$. Note if no feasible solution exists for either of these problems, one vacuously has concavity as
    $$\lambda p_{L,D}(\delta^{(1)}) + (1-\lambda) p_{L,D}(\delta^{(2)}) = -\infty \leq p_{L,D}(\lambda \delta^{(1)}+ (1-\lambda)\delta^{(2)}) \ . $$
    Note without loss of generality, $x^{(i)}_\star=0$ and $f^{(i)}(x^{(i)}_\star)=0$. Then one can easily check $\lambda f^{(1)} + (1-\lambda)f^{(2)}, \lambda x_0^{(1)} + (1-\lambda)x_0^{(2)}, \lambda x_\star^{(1)} + (1-\lambda)x_\star^{(2)}$ is feasible to the problem defining $p_{L,D}(\lambda \delta^{(1)}+ (1-\lambda)\delta^{(2)})$. Consequently, taking the supremum over choices of $f^{(i)},x_0^{(i)}$ gives concavity as
    \begin{align*}
        \lambda p_{L,D}(\delta^{(1)}) + (1-\lambda) p_{L,D}(\delta^{(2)}) \leq p_{L,D}(\lambda \delta^{(1)}+ (1-\lambda)\delta^{(2)}) \ .
    \end{align*}
\end{proof}
}

Generally, the worst-case functions $f$ attaining~\eqref{eq:value-function} can be nontrivial. To find a tractable family of stepsizes for analysis, we focus on ones where this worst-case behavior is no worse than a simple one-dimensional setting:
Consider the one-dimensional (nonsmooth) convex function linearly decreasing from $x_0$ to $x_\star$ and constant thereafter. That is, given $L,D,\delta$, consider the problem instance $ f(x)=\max\{\delta x/D,0\}$, $x_0=D$, and $x_\star=0$. Provided $\delta$ is small enough (i.e., $\delta \leq LD^2/\sum_{i=0}^{t-1} h_i$), the gradient descent iteration $x_{k+1} = x_k - \frac{h_k}{L} 
\nabla f(x_k)$ has
\begin{align}
    \delta_t = \delta_0 - \frac{\sum_{i=0}^{t-1} h_i}{LD^2}\delta_0^2 \label{eq:one-dimensional-example-rate}
\end{align}
where $\delta_k=f(x_k)-f(x_\star)$ denotes the objective gap.
In this example, gradient descent spends all $t$ iterations moving straight forward in a line along the slope. So the descent achieved is just controlled by the objective's slope $\delta_0/D$ (squared) and the total length of steps taken $\sum_{i=0}^{t-1} h_i/L$.

We say that a stepsize pattern $h$ is straightforward if its worst-case behavior (over all smooth functions, as defined in~\eqref{eq:value-function}) is no worse than this one-dimensional piecewise linear setting for $\delta$ small enough.
Formally, we say that a stepsize pattern $h$ {\color{blue} of length $t$} is \emph{straightforward} if for some $\Delta\in(0,1/2]$,
$$ p_{L,D}(\delta)\leq \delta - \frac{\sum_{i=0}^{t-1} h_i}{LD^2}\delta^2 \qquad \forall \delta\in[0,LD^2\Delta] $$
for any $L,D>0$.
Our analysis of gradient descent with long steps then proceeds by:\\
{\it (i) Certifying $h$ is straightforward (able to be computer automated) \hfill see Theorem~\ref{thm:straightforward-certificate},\\
(ii) Solving a resulting recurrence, showing $\delta_{T}\leq \frac{LD^2}{\mathrm{avg}(h) T} + O(1/T^2)$ \hfill see Theorem~\ref{thm:straightforward-rates}.\\
}
The second step is simpler, so we address it first. Additionally, we show the same factor of guarantee improvement $\mathrm{avg}(h)$ carries over to more structured domains like strongly convex optimization.

\subsection{Convergence Guarantees for Straightforward Stepsize Patterns} \label{subsec:rate-proof}
Note straightforwardness provides no guarantees on the intermediate objective values at iterations $k=1,2,\dots t-1$. We show in our theorem below descent at every $t$th iteration. 
To allow for some numerical flexibility, we say a pattern is $\epsilon$\emph{-straightforward} if for some $\Delta\in(0,1/2]$, all $\delta\in[0,LD^2\Delta]$ have value function bounded by $p_{L,D}(\delta) \leq \delta - \frac{\sum_{i=0}^{t-1} (h_i - \epsilon)}{LD^2}\delta^2$.

\begin{theorem}\label{thm:straightforward-rates}
    Consider any $L$-smooth, convex $f$. If $h=(h_0,\dots,h_{t-1})$ is $\epsilon$-straightforward with parameter $\Delta\in (0,1/2]$, then gradient descent~\eqref{eq:pattern-GD} has, for any $s\in\mathbb{N}$, after $T=st$ gradient steps
    \begin{equation} \label{eq:full-rate-statement}
        f(x_{T}) - f(x_\star) \leq 
    \begin{cases}
        (1-\sum_{i=0}^{t-1} (h_i - \epsilon)\Delta)^s(f(x_0)-f(x_\star)) & \text{ if } s \leq \bar s\\
        \dfrac{LD^2}{(\mathrm{avg}(h) - \epsilon)(T-\bar s t) + \frac{1}{\Delta}} & \text{ if } s > \bar s 
    \end{cases}
    \end{equation}
    where $\bar s = \left\lceil \frac{\log\left(\frac{f(x_0)-f(x_\star)}{LD^2\Delta}\right)}{\log(1-\sum_{i=0}^{t-1} (h_i - \epsilon)\Delta)}\right\rceil$. In particular, suppressing lower-order terms, this guarantee is
    $$ f(x_{T}) - f(x_\star) \leq \frac{LD^2}{(\mathrm{avg}(h) -\epsilon) T} + O(1/T^2) \ . $$
\end{theorem}
\begin{proof}
    For any $\epsilon\geq 0$, we begin by showing $\epsilon$-straightforwardness implies that for any $s\in\mathbb{N}$, the recurrence relation
    \begin{equation}\label{eq:straightforward-recurrence}
        \delta_{(s+1)t} \leq
    \begin{cases} \delta_{st} - \frac{\sum_{i=0}^{t-1} (h_i - \epsilon)}{LD^2} \delta_{st}^2 & \text{ if }\delta_{st}\leq LD^2\Delta \\
    \left(1 - \sum_{i=0}^{t-1} (h_i - \epsilon) \Delta\right)\delta_{st} & \text{ if }\delta_{st} > LD^2\Delta \ .
    \end{cases}
    \end{equation}
    Consider first $s=0$. The first case of~\eqref{eq:straightforward-recurrence} follows from the definition of straightforwardness as $\delta_t\leq p_{L,D}(\delta_0)\leq \delta_0 - \frac{\sum_{i=0}^{t-1} (h_i - \epsilon)}{LD^2} \delta_0^2$. {\color{blue} For the second case of~\eqref{eq:straightforward-recurrence}, first note that $p_{L,D}(\cdot)$ is concave by Lemma~\ref{lem:p-concave}.} Then since $p_{L,D}(0)=0$ and $p_{L,D}\left(LD^2\Delta\right) \leq LD^2\left(\Delta - \sum_{i=0}^{t-1} (h_i - \epsilon) \Delta^2\right)$, every $\delta_0 > LD^2\Delta$ must have
    $$ \delta_t \leq p_{L,D}(\delta_0) \leq \frac{\delta_0}{LD^2\Delta}p_{L,D}\left(LD^2\Delta\right) \leq \delta_0\left(1 -\sum_{i=0}^{t-1} (h_i - \epsilon)\Delta\right) \ . $$
    As a result, applying the sequence of steps from the straightforward pattern $h$ from $x_0$ yields $\delta_t\leq \delta_0$. Hence $\|x_{t}-x_\star\|\leq D$ since $x_{t}$ lies in the initial level set $\{x \mid f(x)\leq f(x_0)\}$. Thus the above reasoning can apply inductively, giving the claimed recurrence. 
    
    Next, we show this recurrence implies the claimed convergence guarantee~\eqref{eq:full-rate-statement}. Suppose first, $\delta_{st}$ is larger than $LD^2\Delta$. Then applying the pattern $h$ contracts the objective gap\footnote{Interestingly, this contraction factor is independent of $L,D$. As a result, problem conditioning plays a minimal role in this initial phase of convergence. Instead, only $h$ and the associated straightforwardness parameter $\Delta$ matter.}, inductively giving $\delta_{st} \leq (1-\sum_{i=0}^{t-1} (h_i - \epsilon)\Delta)^s\delta_0$. After at most $\bar s$ executions of the stepsize pattern, one must have $\delta_{st} \leq LD^2\Delta$.
    Afterward, for any $s > \bar s$, the objective gap decreases by at least
    $$ \delta_{(s+1)t} \leq \delta_{st} - \frac{\sum_{i=0}^{t-1} (h_i - \epsilon)}{LD^2} \delta_{st}^2 \ . $$
    Solving this recurrence with the initial condition $\delta_{\bar s t} \leq LD^2\Delta$ gives
    \begin{align*}
        \delta_{st} \leq \frac{LD^2}{\sum_{i=0}^{t-1} (h_i - \epsilon)(s-\bar s) + \frac{1}{\Delta}}\ .
    \end{align*}
\end{proof}

\subsection{Faster Convergence for Straightforward Patterns given Growth Bounds} \label{subsec:rate-proof-SC}
A function $f$ is $\mu$-strongly convex if $f - \frac{\mu}{2}\|\cdot\|^2_2$ is convex. This condition is well-known to lead to linear convergence for most first-order methods in smooth optimization. More generally, faster convergence occurs whenever $f$ satisfies a H\"older growth or error bound condition. Here we consider settings where all $x\in\mathbb{R}^n$ within the level set $f(x)\leq f(x_0)$ satisfy
\begin{equation}\label{eq:Holder-growth}
    f(x) - f(x_\star) \geq \frac{\mu}{q}\|x-x_\star\|^q_2 \ . 
\end{equation}
Strong convexity implies this condition with $q=2$ and leads gradient descent with constant $h=(1,1,\dots)$ to converge at a rate of $O((1-\mu/L)^T)$. When $q>2$, improved sublinear guarantees of $O((L/\mu^{2/q}T)^{q/(q-2)})$ follow. Below, we show that any straightforward stepsize pattern enjoys the same convergence improvements, gaining a similar factor of $\mathrm{avg}(h)$.
{\color{blue} 
\begin{theorem}\label{thm:straightforward-rates-SC}
    Consider any $L$-smooth, convex objective $f$ satisfying~\eqref{eq:Holder-growth}. If $h=(h_0,\dots,h_{t-1})$ is $\epsilon$-straightforward with parameter $\Delta\in(0,1/2]$, then gradient descent~\eqref{eq:pattern-GD} has, for any $T=st$
    $$ f(x_{T}) - f(x_\star) \leq \left(1-\min\left\{\Delta,\ \frac{\mu}{2L}\right\}\sum_{i=0}^{t-1} (h_i - \epsilon) \right)^s(f(x_0)-f(x_\star)) $$
    when $q=2$ and when $q>2$ and $s > \bar s := \left\lceil \frac{\log\left(\frac{f(x_0)-f(x_\star)\mu^{2/(q-2)}}{q^{2/(q-2)}(L\Delta)^{q/(q-2)}}\right)}{\log(1-\sum_{i=0}^{t-1} (h_i - \epsilon)\Delta)}\right\rceil$ has
    $$
        f(x_{T}) - f(x_\star) \leq 
        q\left(\dfrac{L}{\mu^{2/q}(q-2)(\mathrm{avg}(h)-\epsilon)((s-\bar s)t)}\right)^{q/(q-2)} \ . 
    $$
\end{theorem}
\begin{proof}
    Let $D_k = \sup\{\|x-x_\star\| \mid f(x)\leq f(x_k)\}$ denote the size of each level set visited by gradient descent. The growth bound~\eqref{eq:Holder-growth} ensures $D_k \leq (\frac{q}{\mu}\delta_k)^{1/q}$. Then the recurrence~\eqref{eq:straightforward-recurrence} implies
    $$
        \delta_{(s+1)t} \leq
    \begin{cases} \delta_{st} - \frac{\mu^{2/q}\sum_{i=0}^{t-1}(h_i-\epsilon)}{Lq^{2/q}} \delta_{st}^{2-2/q} & \text{ if }\delta_{st}\leq LD^2_{st}\Delta \\
    \left(1 - \sum_{i=0}^{t-1} (h_i - \epsilon) \Delta\right)\delta_{st} & \text{ if }\delta_{st} > LD_{st}^2\Delta \ .
    \end{cases}
    $$
    One can bound this recurrence relation by the maximum of the two cases. If $q=2$, both cases give a contraction and so
    $$ \delta_{st} \leq \left(1-\min\left\{\Delta,\ \frac{\mu}{2L}\right\}\sum_{i=0}^{t-1} (h_i - \epsilon) \right)^s\delta_0 \ . $$    
    If $q>2$, one has
    $$ \delta_{(s+1)t} \leq \max\left\{\delta_{st} - \frac{\mu^{2/q}\sum_{i=0}^{t-1}(h_i-\epsilon)}{Lq^{2/q}} \delta_{st}^{2-2/q},\ \left(1 - \sum_{i=0}^{t-1} (h_i - \epsilon) \Delta\right)\delta_{st} \right\} \ . $$
    Solving this recurrence gives linear convergence until the objective gap is less than $\left(\frac{q^{2/q}L\Delta}{\mu^{2/q}}\right)^{\frac{q}{q-2}}$ is reached, which must occur by iteration $\bar s$, and then afterwards gives a sublinear convergence rate (for example, see~\cite[Lemma A.1]{Diaz2023bundleRates} for this calculation), giving the claim.
\end{proof}

Note that the above bound for the case of $q>2$ is independent of $D$. Asymptotically, this is reasonable as better bounds on distance to optimal are eventually available, namely $(\frac{q}{\mu}\delta_k)^{1/q}$. One could give a bound depending on $D=D_0$ early onby using the stronger bound $D_k \leq \min\{(\frac{q}{\mu}\delta_k)^{1/q}, D_0\}$. However, once the gap is small enough for the first term to dominate, the bound will again be independent of $D_0$.
}
    \section{Certificates of Straightforwardness}\label{sec:certificate}
All that remains is to show how one can certify the straightforwardness of a stepsize pattern.
% Recall proving $\epsilon\geq 0$-straightforwardness amounts to proving an upper bound on the maximum objective gap after $t$ steps of
% $$ p(\delta) \leq \delta - \frac{\sum_{i=0}^{t-1} (h_i - \epsilon)}{LD^2}\delta^2 \qquad \forall \delta\in[0,LD^2\Delta] \ . $$
We do this in two steps. First, Section~\ref{subsec:reform} shows that $p_{L,D}(\delta)$ is upper bounded by an SDP minimization problem using previously developed PEP techniques. Hence straightforwardness is implied by showing the SDP corresponding to each $\delta\in[0,LD^2\Delta]$ has a sufficiently good feasible solution.
Second, Section~\ref{subsec:certificate} shows that another semidefinite programming feasibility problem can certify that such an interval of solutions exists.

\subsection{The Worst-Case Value Function is Upper Bounded by an SDP} \label{subsec:reform}
We first reformulate the infinite-dimensional problem~\eqref{eq:value-function} as a finite-dimensional nonconvex quadratically constrained quadratic problem (see~\eqref{eq:QCQP-form}), relax that formulation into an SDP (see~\eqref{eq:SDP-form}), and then upper bound the SDP by its dual problem (see~\eqref{eq:dual-form}). This process was been developed by~\cite{drori2012PerformanceOF,taylor2017interpolation,taylor2017smooth} and is carried out in our particular setting below.\\[-8pt]

\noindent {\bf Step 1: A QCQP reformulation.} First, as proposed by Drori and Teboulle~\cite{drori2012PerformanceOF}, one can discretize the infinite-dimensional problem defining $p_{L,D}(\delta)$ over all possible objective values $f_k$ and gradients $g_k$ at the points $x_k$ with $k\in I_t^\star := \{\star,0,1,\dots t\}$ as done below. Using the interpolation theorem of Taylor et al.~\cite{taylor2017interpolation}, this is an exact reformulation rather than a relaxation, giving
\begin{equation}
p_{L,D}(\delta) = \begin{cases}
    \max_{x_0,f,g} & f_t - f_\star\\
    \mathrm{s.t.} & f_i \geq f_j + g_j^T(x_i-x_j) + \frac{1}{2L}\|g_i-g_j\|^2_2 \qquad \forall i,j\in I_t^\star \\
    & \|x_0-x_\star\|^2_2 \leq D^2\\
    & f_0 - f_\star \leq \delta \\
    & x_\star=0,f_\star=0, g_\star=0 \\
    & x_{i+1} = x_i - \frac{h_i}{L} g_i \qquad \forall i=0,\dots, t-1
\end{cases}\label{eq:QCQP-form}
\end{equation}
where, without loss of generality, we have fixed $x_\star=0,f_\star=0, g_\star=0$. \\[-8pt]

\noindent {\bf Step 2: An SDP relaxation.} Second, one can relax the nonconvex problem~\eqref{eq:QCQP-form} to the following SDP as done in~\cite{taylor2017interpolation,taylor2017smooth,Drori2018EfficientFM}.
We follow the particular notational choices of~\cite{gupta2023branch}, which recently considered globally, numerically optimizing gradient descent's stepsizes with fixed total iterations, defining
\begin{align*}
H & :=[x_{0}\mid g_{0}\mid g_{1}\mid\ldots\mid g_{t}]\in\mathbb{R}^{d\times(t+2)}\ ,\\
G & :=H^TH\in\mathbb{S}_{+}^{t+2}\ ,\\
F & :=[f_{0}\mid f_{1}\mid\ldots\mid f_{t}]\in\mathbb{R}^{1\times(t+1)} \ , 
\end{align*}
with the following notation for selecting columns and elements of $H$ and $F$: 
\begin{align*}
\mathbf{g}_{\star} & :=0\in\mathbb{R}^{t+2},\;\mathbf{g}_{i}:=e_{i+2}\in\mathbb{R}^{t+2},\quad i\in[0:t]\\
\mathbf{x}_{0} & :=e_{1}\in\mathbb{R}^{t+2},\;\mathbf{x}_{\star}:=0\in\mathbb{R}^{t+2}\ ,\\
\mathbf{x}_{i} & :=\mathbf{x}_{0}-\frac{1}{L}\sum_{j=0}^{i-1}h_j\mathbf{g}_{j}\in\mathbb{R}^{t+2},\quad i\in[1:t]\\
\mathbf{f}_{\star} & :=0\in\mathbb{R}^{t+1},\;\mathbf{f}_{i}:=e_{i+1}\in\mathbb{R}^{t+1},\quad i\in[0:t] \ .
\end{align*}
This notation ensures
$
x_{i}=H\mathbf{x}_{i}$, $g_{i}=H\mathbf{g}_{i}$, and $f_{i}=F\mathbf{f}_{i}.
$
Furthermore,
for $i,j\in I_{t}^{\star}$, define 
\begin{align*}
A_{i,j}(h)& :=\mathbf{g}_{j}\odot(\mathbf{x}_{i}-\mathbf{x}_{j})\in\mathbb{S}^{t+2}\ ,\\
B_{i,j}(h)& :=(\mathbf{x}_{i}-\mathbf{x}_{j})\odot(\mathbf{x}_{i}-\mathbf{x}_{j})\in\mathbb{S}_{+}^{t+2}\ ,\\
C_{i,j}& :=(\mathbf{g}_{i}-\mathbf{g}_{j})\odot(\mathbf{g}_{i}-\mathbf{g}_{j})\in\mathbb{S}_{+}^{t+2}\ ,\\
a_{i,j}& :=\mathbf{f}_{j}-\mathbf{f}_{i}\in\mathbb{R}^{t+1}
\end{align*}
where $x\odot y= \frac{1}{2}(xy^T + yx^T)$ denotes the symmetric outer product. This
notation is defined so that $ g_{j}^T(x_{i}-x_{j}) =\tr GA_{i,j}(h)$, $\|x_{i}-x_{j}\|^{2}_2=\tr GB_{i,j}(h)$, and $\|g_{i}-g_{j}\|^{2}_2=\tr GC_{i,j}$ for any $i,j\in I_{t}^{\star}$.
Then the QCQP formulation~\eqref{eq:QCQP-form} can be relaxed to
\begin{align}
p_{L,D}(\delta)
\leq & \begin{cases}
\max_{F,G} & F\mathbf{f}_{t}\\
\textrm{s.t.} & Fa_{i,j}+\tr GA_{i,j}(h)+\frac{1}{2L}\tr GC_{i,j}\leq0,\quad \forall i,j\in I_{t}^{\star}:i\neq j \\
& -G\preceq 0\\
& \tr GB_{0,\star}\leq D^{2}\\
& F\mathbf{f}_{0} \leq \delta \ .
\end{cases}\label{eq:SDP-form}
\end{align}
{\color{blue} Equality holds above if one adds the constraint that $\mathrm{rank\ } G \leq n$ where $n$ is the dimension of problem instances considered when defining $p_{L,D}$. This constraint is vacuously true if the considered problem dimension $n$ exceeds $t+2$. Consequently, the QCQP~\eqref{eq:QCQP-form} and the SDP~\eqref{eq:SDP-form} are equivalent, provided one allows sufficiently high-dimensional objectives, but may differ if one restricts to finding the worst-case over lower-dimensional problem instances.} However, this equality is not needed for our analysis, so we make no such assumption. %with decision variables $F\in\mathbb{R}^{1\times(t+1)}$ and $G\in\mathbb{R}^{(t+2)\times(t+2)}$. \\[-8pt]

\noindent {\bf Step 3: The upper bounding dual SDP.} Third, note the maximization SDP~\eqref{eq:SDP-form} is bounded above by its dual minimization SDP by weak duality, giving
\begin{equation}
p_{L,D}(\delta) \leq
\begin{cases}
    \min_{\lambda,v,w,Z} & D^{2} v + \delta w \\
\textrm{s.t.} &
\sum_{i,j\in I_{t}^{\star}:i\neq j}\lambda_{i,j}a_{i,j}= a_{\star, t} - w a_{\star, 0} \\
&v B_{0,\star} +\sum_{i,j\in I_{t}^{\star}:i\neq j}\lambda_{i,j}\left(A_{i,j}(h)+\frac{1}{2L}C_{i,j}\right)=Z \\
&Z\succeq0 \\
&v,w\geq0,\;\lambda_{i,j}\geq0,\quad \forall i,j\in I_{t}^{\star}:i\neq j \ .
\end{cases}\label{eq:dual-form}
\end{equation}
Although it is not needed for our analysis, equality holds here as well (i.e., strong duality holds) due to~\cite[Theorem 6]{taylor2017interpolation}.

\subsection{An SDP Feasibility Certificate that implies Straightforwardness}\label{subsec:certificate}
The preceding bound~\eqref{eq:dual-form} establishes that $\epsilon\geq 0$-straightforwardness holds if for some $\Delta \in (0,1/2]$, every $\delta\in [0,LD^2\Delta]$ has a corresponding dual feasible solution with objective at most $\delta - \frac{\sum_{i=0}^{t-1} (h_i - \epsilon)}{LD^2} \delta^2$.
We claim that it suffices to fix $L=1,D=1$ without loss of generality. For any $L$-smooth $f$ with $\|x_0-x_\star\|\leq D$, this follows by instead considering minimizing $\tilde f(\tilde x) = \frac{1}{LD^2}f(D\tilde x)$. One can easily verify $\tilde f$ is $1$-smooth, has $\|\tilde x_0 - \tilde x_\star\|\leq 1$ for $\tilde x_0 = Dx_0, \tilde x_\star = Dx_\star$, and gradient descent $\tilde x_{k+1} = \tilde x_{k} - h_k \nabla \tilde f(x_k)$ produces exactly the iterates of $x_{k+1}=x_k - h_k/L\nabla f(x_k)$ rescaled by $D$. Hence $p_{L,D}(\delta) = LD^2 p_{1,1}(\delta/LD^2)$.

We restrict our search for dual certificates bounding $p_{1,1}(\delta)$ to a special case, which we numerically observed to hold approximately at the minimizers of~\eqref{eq:dual-form}: given $\delta$, fix $v=\sum_{i=0}^{t-1} (h_i + \epsilon)\delta^2$ and $w = 1 - 2\sum_{i=0}^{t-1} h_i \delta$. Noting this fixed variable setting has $v + \delta w = \delta - \sum_{i=0}^{t-1} (h_i - \epsilon)\delta^2$, $\epsilon$-straightforwardness follows if one can show feasible solutions with these fixed values exist.

Given $\delta$ and a selection of $\lambda\in\mathbb{R}^{(t+2)\times(t+2)}$ and fixing $v$ as above, we define
\begin{align}
    Z_{h,\epsilon}(\lambda, \delta) := \sum_{i=0}^{t-1} (h_i + \epsilon)\delta^2 B_{0,\star} +\sum_{i,j\in I_{t}^{\star}:i\neq j}\lambda_{i,j}\left(A_{i,j}(h)+\frac{1}{2}C_{i,j}\right) \ .
\end{align}
Observe that $Z_{h,\epsilon}(\lambda,\delta)$ is nearly linear: the first entry has the only nonlinear behavior, depending quadratically on $\delta$, with the rest depending only linearly on $\lambda$. Written in block form, we denote
\begin{align}
    Z_{h,\epsilon}(\lambda,\delta) =: \begin{bmatrix} \sum_{i=0}^{t-1} (h_i + \epsilon)\delta^2 & m_h(\lambda)^T \\ m_h(\lambda) & M_h(\lambda) \end{bmatrix} \label{eq:block-form}
\end{align}
where $m_h \colon \mathbb{R}^{(t+2)\times(t+2)} \rightarrow \mathbb{R}^{t+1}$ and $M_h \colon \mathbb{R}^{(t+2)\times(t+2)}\rightarrow \mathbb{R}^{(t+1)\times(t+1)}$ are linear functions. Certifying $p_{1,1}(\delta)\leq \delta - \sum_{i=0}^{t-1} (h_i - \epsilon)\delta^2$ for fixed $\delta$ then follows by showing the following spectral set is nonempty
\begin{align}
    \mathcal{R}_{h,\epsilon,\delta} = \left\{ \lambda\in\mathbb{R}^{(t+2)\times(t+2)} \mid \begin{array}{l}\sum_{i,j\in I_{t}^{\star}:i\neq j}\lambda_{i,j}a_{i,j}= a_{\star, t} - \left(1 - 2\sum_{i=0}^{t-1} h_i\delta\right)a_{\star, 0}\\ \lambda \geq 0\\ Z_{h,\epsilon}(\lambda,\delta)\succeq 0\end{array}\right\} \ . \nonumber
\end{align}
\begin{lemma}\label{lem:single-delta}
    A stepsize pattern $h\in\mathbb{R}^t$ is $\epsilon\geq 0$-straightforward if for some $\Delta\in (0,1/2]$, $\mathcal{R}_{h,\epsilon,\delta}$ is nonempty for all $\delta\in[0,\Delta]$. Straightforwardness of $h$ is implied by each $\mathcal{R}_{h,0,\delta}$ being nonempty.
\end{lemma}

This lemma alone does not directly enable the computation of a convergence-proof certificate. One would need certificates of feasibility for the infinitely many sets given by each $\delta\in [0,\Delta]$. The following theorem shows that the existence of such solutions can be certified via a single feasible solution to yet another semidefinite program.
\begin{theorem} \label{thm:straightforward-certificate}
    A stepsize pattern $h\in\mathbb{R}^t$ is $\epsilon\geq 0 $-straightforward if for some $\Delta\in (0,1/2]$, $\mathcal{S}_{h,\epsilon,\Delta}$ is nonempty where
    $$ \mathcal{S}_{h,\epsilon,\Delta} = \left\{ (\lambda,\gamma)\in\mathbb{R}^{(t+2)\times (t+2)}\times\mathbb{R}^{(t+2)\times (t+2)} \mid \begin{array}{l}
    \sum_{i,j\in I_{t}^{\star}:i\neq j}\lambda_{i,j}a_{i,j}= a_{\star, t} - a_{\star, 0}\\
    \sum_{i,j\in I_{t}^{\star}:i\neq j}\gamma_{i,j}a_{i,j}= 2\sum_{i=0}^{t-1} h_i a_{\star, 0}\\
    m_h(\lambda)=0 \\
    \lambda \geq 0, \lambda + \Delta\gamma \geq 0\\
    \begin{bmatrix} \sum_{i=0}^{t-1} (h_i + \epsilon) & m_h(\gamma)^T \\ m_h(\gamma) & M_h(\lambda) \end{bmatrix}\succeq 0\\
    \begin{bmatrix} \sum_{i=0}^{t-1} (h_i + \epsilon) & m_h(\gamma)^T \\ m_h(\gamma) & M_h(\lambda+\Delta\gamma) \end{bmatrix}\succeq 0 
    \end{array}\right\} \ .$$
\end{theorem}
\begin{proof}
    Let $(\lambda,\gamma) \in \mathcal{S}_{h,\epsilon,\Delta}$. We prove this by showing $\lambda^{(\delta)}:=\lambda+\delta\gamma \in \mathcal{R}_{h,\epsilon,\delta}$ for every $\delta\in[0,\Delta]$ by Lemma~\ref{lem:single-delta}. This amounts to verifying the three conditions defining $\mathcal{R}_{h,\epsilon,\delta}$ for each $\lambda^{(\delta)}$. 

    First, we check $\sum_{i,j\in I_{t}^{\star}:i\neq j}\lambda^{(\delta)}_{i,j}a_{i,j}= a_{\star, t} - \left(1 - 2\sum_{i=0}^{t-1} h_i\delta\right)a_{\star, 0}$. The first equality defining $\mathcal{S}_{h,\epsilon,\Delta}$ ensures this for $\lambda^{(0)}=\lambda$. Adding $\delta$ times the second equality defining $\mathcal{S}_{h,\epsilon,\Delta}$ establishes the equality for every $\lambda^{(\delta)}$  as $\sum_{i,j\in I_{t}^{\star}:i\neq j}(\lambda^{(\delta)}_{i,j}+\delta\gamma_{i,j})a_{i,j}= a_{\star, t} - \left(1 - 2\sum_{i=0}^{t-1} h_i\delta\right)a_{\star, 0}$.
    Second, we check nonnegativity $\lambda^{(\delta)}\geq 0$. This follows by noting $\lambda^{(\delta)}$ is a convex combination of $\lambda$ and $\lambda+\Delta\gamma$, which are nonnegative by construction.    
    Finally, we check the nonlinear (but nearly linear) condition $Z_{h,\epsilon}(\lambda^{(\delta)},\delta)\succeq 0$. We consider the block-form~\eqref{eq:block-form} of this semidefinite inequality, seeking
    $$ Z_{h,\epsilon}(\lambda^{(\delta)},\delta) = \begin{bmatrix} \sum_{i=0}^{t-1} (h_i + \epsilon)\delta^2 & m_h(\lambda^{(\delta)})^T \\ m_h(\lambda^{(\delta)})) & M_h(\lambda^{(\delta)}) \end{bmatrix} \succeq 0 \ . $$
    Since $m_h(\lambda)=0$, using the linearity of $m_h$ and $M_h$, the above can be expanded to equal
    $$\begin{bmatrix} \sum_{i=0}^{t-1} (h_i + \epsilon)\delta^2 & \delta m_h(\gamma)^T \\ \delta m_h(\gamma) & M_h(\lambda) + \delta M_h(\gamma) \end{bmatrix} \succeq 0\ .
    $$
    Rescaling the first row and column by $1/\delta$ gives an equivalent condition, which is now linear in $\delta$,
    $$ \begin{bmatrix} \sum_{i=0}^{t-1} (h_i + \epsilon) & m_h(\gamma)^T \\ m_h(\gamma) & M_h(\lambda) + \delta M_h(\gamma) \end{bmatrix} \succeq 0\ .
    $$
    When $\delta=0$ or $\Delta$, this condition is explicitly ensured by the definition of $\mathcal{S}_{h,\epsilon,\Delta}$. Then the linearity and convexity of this condition imply it holds for all intermediate $\lambda^{(\delta)}$, completing the proof.
\end{proof}
{\color{blue} Note computing a member of $\mathcal{S}_{h,\epsilon,\Delta}$ does not correspond to solving a particular performance estimation problem. Instead, each member provides a ``line segment'' of solutions to a series of the performance estimation problems $p_{L,D}(\delta)$ for an interval of possible $\delta$ values.}

\subsection{Certificates of Straightforwardness Proving Guarantees in Table~\ref{tab:patterns-and-rates}}
To prove a given pattern $h$ converges at rate $LD^2/\mathrm{avg}(h)T$, we only need to show some $\mathcal{S}_{h,0,\Delta}$ is nonempty. The most natural path is to provide an exact member of this set. For all of the stepsizes in Table~\ref{tab:patterns-and-rates}, it was relatively easy to find a feasible solution $\mathcal{S}_{h,0,\Delta}$ in floating point arithmetic via an interior point method. However, exactly identifying a member of $\mathcal{S}_{h,0,\Delta}$ from this can still be hard.

First, we prove the claimed guarantees for the $t=2$ and $t=3$ stepsize patterns of Table~\ref{tab:patterns-and-rates} by presenting exact members of $\mathcal{S}_{h,0,\Delta}$. Then, to handle larger values of $t$, we present a simple rounding approach able to produce members of $\mathcal{S}_{h,\epsilon,\Delta}$, {\color{blue} often with $\epsilon$ around the accuracy of our SDP solves $\approx 10^{-9}$}. This approach produced rational-valued certificates proving the rest of the claimed convergence guarantees in Table~\ref{tab:patterns-and-rates}.

The exact rational arithmetic verifying the correctness of all certificates $(\lambda,\gamma)$ was done in \texttt{Mathematica 13.0.1.0}. Note that the entries in these certificates for $t\geq 7$ are entirely computer-generated and lack real human insight. As an example for reference, the certificate for $t=7$ is included in the appendix. Larger certificates are impractical to include here. For example, our $t=127$ guarantee is certificate $(\lambda,\gamma)$ has $32640$ nonzero entries. Certificates for every pattern in Table~\ref{tab:patterns-and-rates} and exact verifying computations are available at~\url{github.com/bgrimmer/LongStepCertificates}.

\begin{theorem}\label{thm:two-step-rate}
    For any $\eta \in (0,3)$, the stepsize pattern $h=(3-\eta,1.5)$ is straightforward. Hence gradient descent~\eqref{eq:pattern-GD} alternating between these two stepsizes has every even $T$ satisfy $$ f(x_T)-f(x_\star) \leq \frac{LD^2}{(2.25-\eta/2)\times T} + O(1/T^2) \ .$$
\end{theorem}
\begin{proof}
    For any $\eta\in(0,3)$, consider the selection of $(\lambda,\gamma)$ given by
\begin{align*}
     \lambda =
\begin{bmatrix}
0 & 0 & 0 & 0  \\
0 & 0 & \frac{1}{2} & \frac{1}{2}  \\
0 & 0 & 0 & \frac{1}{2}  \\
0 & 0 & 0 & 0  
\end{bmatrix} \ , \qquad
    \gamma =
\begin{bmatrix}
0 & 3-\eta & \frac{6-\eta}{2} & \frac{6-\eta}{2} \\
0 & 0 & \frac{-(6-\eta)}{2} & \frac{-(6-\eta)}{2} \\
0 & 0 & 0 & 0 \\
0 & 0 & 0 & 0
\end{bmatrix} \ . 
\end{align*}
It suffices to show for some $\Delta\in (0,1/2]$, $(\lambda,\gamma)\in \mathcal{S}_{h,0,\Delta}$. One can easily verify the needed equalities and nonnegativities hold for all $0 \leq \Delta\leq 1/(6-\eta)$. The first positive semidefiniteness condition of $(\lambda,\gamma)\in \mathcal{S}_{h,0,\Delta}$ amounts to checking every $\eta\in (0,3)$ has
$$ \begin{bmatrix} \sum_{i=0}^{t-1} (h_i + \epsilon) & m_h(\gamma)^T \\ m_h(\gamma) & M_h(\lambda) \end{bmatrix} =
\begin{bmatrix}
\frac{9-2\eta}{2} & \frac{-(3-\eta)}{2} & \frac{-(6-\eta)}{4} & \frac{-(6-\eta)}{4}  \\
\frac{-(3-\eta)}{2} & \frac{1}{2} & \frac{2-\eta}{4} & \frac{2-\eta}{4}  \\
\frac{-(6-\eta)}{4} & \frac{2-\eta}{4} & \frac{1}{2} & \frac{1}{2}  \\
\frac{-(6-\eta)}{4} & \frac{2-\eta}{4} & \frac{1}{2} & \frac{1}{2} 
\end{bmatrix} \succeq 0  \ .$$
Since this convex condition is linear in $\eta$, it suffices to check it at $\eta=0$ and $\eta=3$. Moreover, for any $\eta\in(0,3)$, note this matrix has exactly two zero eigenvalues with associated eigenvectors spanning $(1/2,1/2,1,0)$ and $(1/2,1/2,0,1)$. The second positive semidefiniteness condition amounts to checking an update to this matrix of size $\Delta$ remains positive semidefinite, namely $\begin{bmatrix} \sum_{i=0}^{t-1} (h_i + \epsilon) & m_h(\gamma)^T \\ m_h(\gamma) & M_h(\lambda+\Delta\gamma) \end{bmatrix} $ which equals
$$
\begin{bmatrix}
\frac{9-2\eta}{2} & \frac{-(3-\eta)}{2} & \frac{-(6-\eta)}{4} & \frac{-(6-\eta)}{4}  \\
\frac{-(3-\eta)}{2} & \frac{1}{2} & \frac{2-\eta}{4} & \frac{2-\eta}{4}  \\
\frac{-(6-\eta)}{4} & \frac{2-\eta}{4} & \frac{1}{2} & \frac{1}{2}  \\
\frac{-(6-\eta)}{4} & \frac{2-\eta}{4} & \frac{1}{2} & \frac{1}{2} 
\end{bmatrix} + \begin{bmatrix}
0 & 0 & 0 & 0  \\
0 & \frac{-3}{2} & \frac{6-\eta}{4} & \frac{6-\eta}{4}  \\
0 & \frac{6-\eta}{4} & 0 & 0  \\
0 & \frac{6-\eta}{4} & 0 & 0
\end{bmatrix}\Delta
$$
must be positive semidefinite.
One can check this added matrix term is positive semidefinite on the subspace spanned by $(1/2,1/2,1,0)$ and $(1/2,1/2,0,1)$ (again by checking when $\eta=0$ and $\eta=3$ and then using convexity). As a result, positive semidefiniteness is maintained for $\Delta$ small enough. Exact arithmetic verifying all of these claims are given in the associated \texttt{Mathematica} notebook.
Hence $(\lambda,\gamma)\in \mathcal{S}_{h,0,\Delta}$, proving the main claim by Theorem~\ref{thm:straightforward-certificate} and the claimed convergence guarantee by Theorem~\ref{thm:straightforward-rates}.
\end{proof}
\begin{theorem}
    The stepsize pattern $h=(1.5,4.9,1.5)$ is straightforward. Hence gradient descent~\eqref{eq:pattern-GD} cycling through these three stepsizes has every $T=3s$ satisfy 
    $$ f(x_T)-f(x_\star) \leq \frac{LD^2}{2.63333... \times T} + O(1/T^2) \ .$$
\end{theorem}
\begin{proof}
    This result is certified with $\Delta=10^{-4}$ by the following exact values for $(\lambda,\gamma)\in \mathcal{S}_{h,0,\Delta}$ of
\begin{align*}
     \lambda =
\begin{bmatrix}
0 & 0 & 0 & 0 & 0 \\
0 & 0 & 1.95 & 0.003 & 0.007 \\
0 & 0.95 & 0 & 0.5 & 0.5 \\
0 & 0.006 & 0 & 0 & 0.51 \\
0 & 0.004 & 0 & 0.013 & 0
\end{bmatrix}, \
    \gamma  =
\begin{bmatrix}
0 & 0.005 & 7.825 & 3.9497 & 4.0203 \\
0 & 0 & -5.24 & -10.555 & 0 \\
0 & 0 & 0 & 7.9 & -5.315 \\
0 & 0 & 0 & 0 & 1.2947 \\
0 & 0 & 0 & 0 & 0
\end{bmatrix}.
\end{align*}
\end{proof}
\noindent Based on numerical exploration, we conjecture that patterns of the form $(3-\eta,1.5)$ are the longest straightforward patterns of length two and $(1.5,5-\eta,1.5)$ are the longest of length three.

For larger settings $t\geq 7$, determining an exact member of $\mathcal{S}_{h,0,\Delta}$ proved difficult. So we resort to a fully automated construction of a convergence guarantee certificate by first numerically computing an approximate member of $\mathcal{S}_{h,0,\Delta}$ (via an interior point method) and then rounding to a nearby rational-valued exact member of $\mathcal{S}_{h,\epsilon,\Delta}$ for some small $\epsilon$. In light of our Theorem~\ref{thm:straightforward-rates}, such rounding only weakens the resulting guarantee's coefficient from $\mathrm{avg}(h)$ to $\mathrm{avg}(h)-\epsilon$.\\[-8pt]

\noindent {\bf Computer Generation of Convergence Proof Certificates.} {\it Given $h$ and $\Delta$, \\
(i) Numerically compute some $(\tilde\lambda,\tilde\gamma)$ approximately in $\mathcal{S}_{h,0,\Delta}$, \\
(ii) Compute rational $(\hat\lambda,\hat\gamma)$ near $(\tilde\lambda,\tilde\gamma)$ exactly satisfying the three needed equalities,\\
(iii) Check in exact arithmetic nonnegativity and positive definiteness of $M_h(\hat\lambda)$ and $M_h(\hat\lambda+\Delta\hat\gamma)$,\\
(iv) If so, $(\hat\lambda,\hat\gamma)\in \mathcal{S}_{h,\epsilon,\Delta}$, certifying $\frac{LD^2}{(\mathrm{avg}(h)-\epsilon)T} + O(1/T^2)$ convergence, for}
$$ \epsilon = \frac{\max\{m_h(\hat\gamma)^TM_h(\hat\lambda)^{-1}m_h(\hat\gamma),\ m_h(\hat\gamma)^TM_h(\hat\lambda+\Delta\hat\gamma)^{-1}m_h(\hat\gamma)\}}{t}-\mathrm{avg}(h) \ .$$
The above value of $\epsilon$ is the smallest value with $(\hat\lambda,\hat\gamma)\in \mathcal{S}_{h,\epsilon,\Delta}$, since by considering their Schur complements, the two needed positive semidefinite conditions hold if and only if
$$ \sum_{i=0}^{t-1}(h_i+\epsilon) - m_h(\hat\gamma)^TM_h(\hat\lambda)^{-1}m_h(\hat\gamma) \geq 0 \ \text{ and } \sum_{i=0}^{t-1}(h_i+\epsilon) - m_h(\hat\gamma)^TM_h(\hat\lambda+\Delta\hat\gamma)^{-1}m_h(\hat\gamma) \geq 0\ .$$

\begin{theorem}
    The stepsize patterns of lengths $t\in \{7,15,31,63,127\}$ in Table~\ref{tab:patterns-and-rates} are all $\epsilon$-straightforward for $\epsilon\in \{10^{-9},10^{-9},10^{-11},10^{-3},10^{-4}\}$ with $\Delta\in\{10^{-5},10^{-6},10^{-8},10^{-7},10^{-8}\}$. Hence the convergence guarantees claimed in Table~\ref{tab:patterns-and-rates} hold for each corresponding ``long step'' gradient descent method.
\end{theorem}
\begin{proof}
    Certificates $(\hat\lambda,\hat\gamma)\in\mathcal{S}_{h,\epsilon,\Delta}$ (produced via the above procedure) with exact arithmetic validation are available at \url{github.com/bgrimmer/LongStepCertificates}.
\end{proof}
    \section{Future Directions}\label{sec:conclusion}
{\color{blue} Building on existing performance estimation ideas, we have demonstrated an analysis technique capable of proving convergence of gradient descent using nonconstant, long stepsize patterns, which gain in strength as longer patterns are considered. This runs contrary to widely held intuitions regarding constant stepsize selections and the importance of monotone objective decreases. Instead, we show that long-run performance can improve by periodically taking (very) long steps that may increase the objective value in the short term. We accomplish this via computer-generated proof certificates bounding the effect of many steps collectively by providing solutions for a sequence of related performance estimation problems.} We conclude by discussing a few possible future improvements on and shortcomings of this technique.\\[-8pt]

{\color{blue}
\noindent {\bf Possible Extension to More Families of Gradient Methods.}
    Often, analysis techniques for gradient descent and its accelerated variants extend rather directly to constrained minimization or minimizing composite objectives $f(x) + r(x)$ by utilizing projections and proximal operators. PEP techniques naturally extend to these settings. Drori~\cite[Theorem 2.7 and 2.9]{Drori2014} shows the tight convergence rate for projected gradient descent is strictly worse than the unconstrained setting, having rate $LD^2/4T$ instead of~\eqref{eq:distance-induction-style-result}. A similar (small) worsening of the optimal rate also holds for accelerated proximal gradient methods~\cite{jang2023computerassisted}. This difficulty appears to extend to our straightforward analysis technique. Figure~\ref{fig:straightforwards} shows the length two and three stepsize patterns that numerically satisfy a generalized notion of straightforwardness for gradient descent, projected gradient descent, and proximal gradient descent. The set of ``straightforward'' patterns for constrained is strictly smaller and composite yet strictly smaller. Consequently, strictly less aggressive stepsizes than  Table~\ref{tab:patterns-and-rates} would be required for future projected or proximal extensions.

    The PEP framework has previously been successfully employed in handling settings of inexact gradients~\cite{Klerk2020} and relatively smooth optimization (via Bregman divergences)~\cite{Dragomir2019OptimalCA}. Determining the degree to which utilizing periodic long steps can improve the convergence of inexact, Bregman, and stochastic variations of gradient descent also provides interesting future directions.

\begin{figure}
    \includegraphics[width=0.33\linewidth]{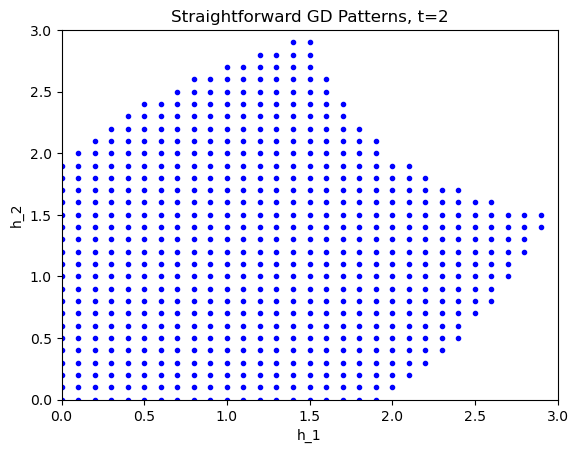}\includegraphics[width=0.33\linewidth]{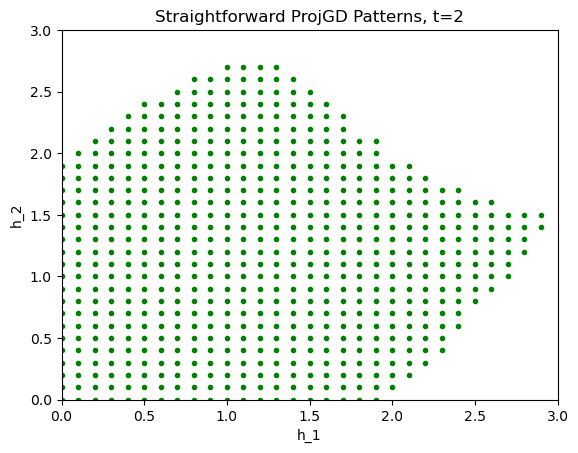}\includegraphics[width=0.33\linewidth]{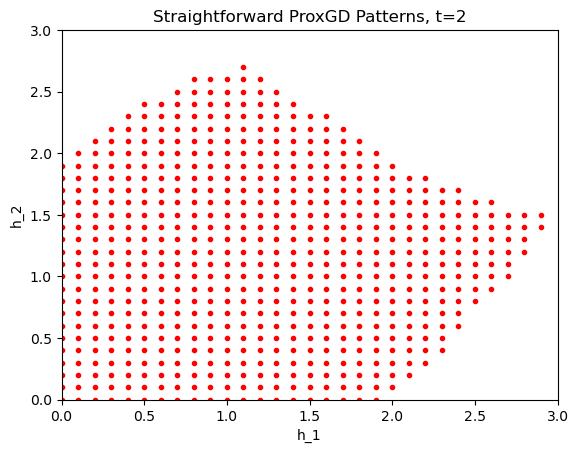}
    \includegraphics[width=0.33\linewidth]{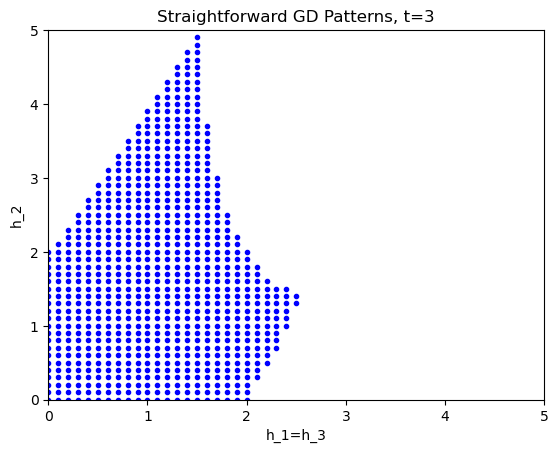}\includegraphics[width=0.33\linewidth]{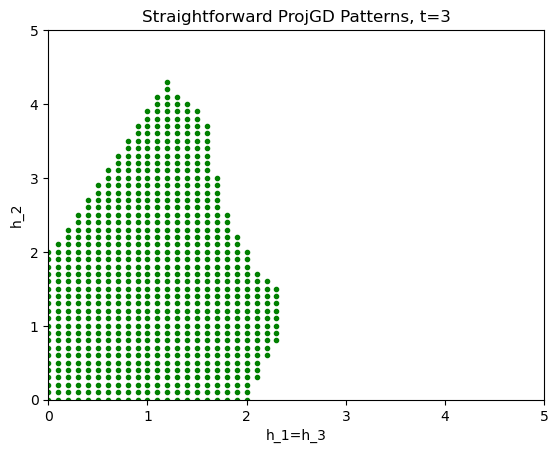}\includegraphics[width=0.33\linewidth]{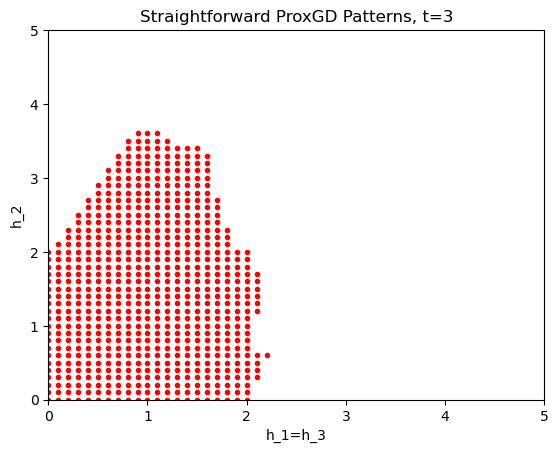}
    \caption{\color{blue} For unconstrained, constrained, and composite minimization, the stepsize patterns satisfying a natural generalization of straightforwardness are shown. Patterns $(h_1,h_2)$ of length two and symmetric patterns $(h_1,h_2,h_1)$ of length three were sampled at every $0.1$ increment. Straightforwardness was approximated by solving a performance estimation problem determining if $f(x_t) - p_\star$ (or $f(x_t) + r(x_t)- p_\star$ for composite problems) is always at most $f(x_0) - p_\star - \frac{\sum h_i}{LD^2} (f(x_0) - p_\star)^2$ (or $f(x_0) +r(x_0)- p_\star - \frac{\sum h_i}{LD^2} (f(x_0) +r(x_0)- p_\star)^2$) given the initial gap was at most $\Delta = 10^{-4}$.}\label{fig:straightforwards}
\end{figure}
}

\noindent {\bf Future Improvements in Algorithm Design.}
    The search for long, straightforward stepsize patterns $h$ is hard. The patterns presented in Table~\ref{tab:patterns-and-rates} resulted from substantial brute force searching. The task of maximizing $\mathrm{avg}(h)$ subject to $h$ being straightforward, although nonconvex, may be approachable using branch-and-bound techniques similar to those recently developed by Gupta et al.~\cite{gupta2023branch} and applied to a range of PEP parameter optimization problems. Such an approach may yield numerically, globally optimal $h$ for fixed length $t$. This may also generate insights into the general form of the longest straightforward stepsize patterns for each fixed $t$.

    One practical drawback of the method~\eqref{eq:pattern-GD} is the requirement that one knows $L$. Since our analysis in Theorem~\ref{thm:straightforward-rates} only relies on decreases in objective value after $t$ steps, backtracking linesearch schemes or other adaptive ideas may be applicable. For example, one could consider {\color{blue} an Armijo-type linesearching procedure, backtracking with an estimate of $\tilde L$ seeking $f(x_t) \leq f(x_0) - \frac{\sum h_i}{\tilde L}\|\nabla f(x_0)\|^2$. Since convexity ensures $\|\nabla f(x_0)\|^2 \geq \delta_0^2/D^2$, this condition guarantees the descent required by straightforwardness. Analysis and practical development of such ideas are beyond our scope.}\\[-8pt]

\noindent {\bf Future Improvements in Analysis Techniques.}
    Future works may improve our analysis by considering other Lyapunov functions. The distance to optimal and norm of the gradient were both used in~\cite[Chapter 8]{altschuler2018greed}. Our proofs are only concerned with the eventual decrease of the objective gap. The analysis of optimal accelerated and subgradient methods relies additionally on the decreasing distance to a minimizer or a decreasing combination thereof. Identifying better Lyapunovs and stepsize patterns guaranteed to eventually decrease them may lead to stronger guarantees. \\[-8pt]

\noindent {\bf Future Improvements in Computational Aspects.}
    We observed that numerically computed primal optimal solutions to~\eqref{eq:SDP-form} were rank-one for all considered straightforward patterns. This corresponds to the worst-case objective function being essentially one-dimensional. {\color{blue} This is in line with many prior PEP-based analyses finding one-dimensional Huber functions often attain the worst-case performance.} This property was not used herein but could likely be leveraged to enable customized solvers for evaluating $p_{L,D}(\delta)$ and checking membership of $\mathcal{S}_{h,0,\Delta}$. Such improvements in tractability for SDPs with rank-one solutions have been studied widely~\cite{Bellavia2019ARI,ding2021storageOptimal,Yurtsever2021scalableSDP,Souto2022operatorSDP,wang2023GTRS,ding2023,wang2022accelerated} and may enable the search for longer, provably faster straightforward stepsize patterns than shown here. 

    As another avenue of improvement, note that any certificates produced by using floating point arithmetic followed by a rounding step (as done here) will likely lose a small $\epsilon$ amount in the guarantee. The use of an algebraic solver, like SPECTRA~\cite{Henrion2016SPECTRAA}, could enable the automated production of exact certificates of straightforwardness as well as being able to certify when $\mathcal{S}_{h,0,\Delta}$ is empty.

    {\small
    \paragraph{Acknowledgements.} The author thanks Jason Altschuler for conversations at the Simons Insitute in Spring 2017 about the preliminary ideas subsequently developed in his excellent Master's thesis work, Fanghua Chen for preliminary discussions based on numerical observations that motivated this work, and Adrien Taylor and Axel B\"ohm for providing useful feedback on the initial presentation of this work. The openly released Performance Estimation Problem Branch-and-Bound software of~\cite{gupta2023branch} was especially helpful in early explorations and building intuitions.
    
    \bibliographystyle{unsrt}
    \bibliography{references}
    }
    \appendix
    \section{A Computer-Generated Straightforwardness Certificate with $\Delta= 10^{-5}$ and $\epsilon=10^{-9}$ for $h=(1.5, 2.2, 1.5, 12.0, 1.5, 2.2, 1.5)$ }
Below is a certificate $(\hat\lambda,\hat\gamma)\in\mathcal{S}_{h,\epsilon,\Delta}$, completely computer generated, proving a $LD^2/(3.1999999\times T)$ rate for the pattern of length $t=7$ in Table~\ref{tab:patterns-and-rates}. Given the length of these $9\times 9$ matrices, we display their first five and last four columns separately below. Exact calculations verifying the feasibility of these values are given in the associated publicly posted \texttt{Mathematica} notebook.
$$\hat\lambda_{1:5} = \begin{bmatrix}
0 & 0 & 0 & 0 & 0  \\
0 & 0 & \frac{8837407518919583}{4503599627370496} & \frac{5370688140802311}{2305843009213693952} & \frac{960254226721649}{144115188075855872}  \\
0 & \frac{2191522964335457}{2251799813685248} & 0 & \frac{4118290538555273}{2251799813685248} & \frac{2688409565283275}{4503599627370496}  \\
0 & \frac{6721678331720401}{2305843009213693952} & \frac{4407991053556385}{18014398509481984} & 0 & \frac{5801123626984329}{1125899906842624}  \\
0 & \frac{1313681189860411}{1152921504606846976} & \frac{2695784755734549}{2251799813685248} & \frac{8068866010524833}{2251799813685248} & 0  \\
0 & \frac{1137203495150241}{4611686018427387904} & \frac{4688926225212825}{9223372036854775808} & \frac{4572972758977097}{4611686018427387904} & \frac{1182579142099203}{1152921504606846976}  \\
0 & \frac{5789750435206701}{18446744073709551616} & \frac{4956986009057139}{9223372036854775808} & \frac{5204247088958345}{4611686018427387904} & \frac{5375927246703545}{9223372036854775808}  \\
0 & \frac{7965115934238233}{36893488147419103232} & \frac{6653418691702949}{9223372036854775808} & \frac{4849791056907609}{4611686018427387904} & \frac{7288685819438951}{9223372036854775808} \\
0 & \frac{1977810093374139}{9223372036854775808} & \frac{1222316311137735}{1152921504606846976} & \frac{6194623724653895}{4611686018427387904} & \frac{1504051934577545}{1152921504606846976}  \\
\end{bmatrix}, $$
$$\hat\lambda_{6:9} = \begin{bmatrix}
 0 & 0 & 0 & 0 \\
 \frac{4697224493034383}{9223372036854775808} & \frac{8368394272075953}{2305843009213693952} & \frac{6128278626086993}{9223372036854775808} & \frac{80543723501136599}{36893488147419103232} \\
 \frac{7733370871946025}{4611686018427387904} & \frac{5388133457257119}{2305843009213693952} & \frac{8119439550484479}{4611686018427387904} & \frac{19797186857495837}{9223372036854775808} \\
 \frac{2740682573240027}{576460752303423488} & \frac{5290333777740781}{1152921504606846976} & \frac{4877381506142205}{1152921504606846976} & \frac{12485929602009775}{2305843009213693952} \\
 \frac{8956652459407733}{72057594037927936} & \frac{4723166920241497}{18014398509481984} & \frac{5356440120675597}{18014398509481984} & \frac{339319020784307309}{1152921504606846976} \\
 0 & \frac{772844649199827}{4503599627370496} & \frac{2338413531710477}{288230376151711744} & \frac{13381432557675475}{2305843009213693952} \\
 \frac{6929007831194361}{288230376151711744} & 0 & \frac{8173541145090013}{36028797018963968} & \frac{3837009581398546887}{18446744073709551616} \\
 \frac{3923691798295005}{144115188075855872} & \frac{1250438084247729}{288230376151711744} & 0 & \frac{18985683012247614283}{36893488147419103232} \\
 \frac{13037057806369}{2251799813685248} & \frac{7367859533233689}{576460752303423488} & \frac{2876396710542189}{288230376151711744} & 0 \\
\end{bmatrix}, $$
$$ \hat\gamma_{1:5} = \begin{bmatrix}
0 & \frac{1445047782665419}{360287970189639680} & \frac{4403953050470099}{36028797018963968} & \frac{12063929807726837}{144115188075855872} & \frac{3158567943322891003}{144115188075855872}  \\
0 & 0 & -\frac{4876214136104831}{140737488355328} & -\frac{8833525210676647}{1125899906842624} & -\frac{2173342829362743}{281474976710656}  \\
0 & -\frac{3452744755754301}{70368744177664} & 0 & -\frac{6337493011708677}{36028797018963968} & -\frac{6175482448055731}{2251799813685248} \\
0 & \frac{5865189115672077}{36028797018963968} & -\frac{7694464236006067}{281474976710656} & 0 & -\frac{5761085953121451}{144115188075855872}  \\
0 & \frac{8990085085489805}{562949953421312} & \frac{8180724221663923}{9007199254740992} & -\frac{170650240960227}{70368744177664} & 0 \\
0 & \frac{814636363991245}{2251799813685248} & -\frac{5002382223044879}{9007199254740992} & -\frac{5824324289872487}{4503599627370496} & -\frac{5420042187723199}{2251799813685248}  \\
0 & \frac{3576868018334213}{18014398509481984} & -\frac{3095322656598895}{18014398509481984} & -\frac{4741216870166437}{4503599627370496} & \frac{8742107678118927}{18014398509481984} \\
0 & \frac{7467009600885335}{72057594037927936} & \frac{7502387060304591}{18014398509481984} & -\frac{2299019529082251}{2251799813685248} & \frac{5226016391129105}{4503599627370496}  \\
0 & \frac{5171454268783955}{18014398509481984} & \frac{4976089659881855}{4503599627370496} & -\frac{849946842273963}{1125899906842624} & \frac{4041378073126239}{2251799813685248}  \\
\end{bmatrix}, $$
$$ \hat\gamma_{6:9} = \begin{bmatrix}
 \frac{842771278878770475}{288230376151711744} & \frac{1763107326300397405}{288230376151711744} & \frac{490223898427757151}{72057594037927936} & \frac{246036905476920105}{36028797018963968} \\
-\frac{7575374264747013}{1125899906842624} & -\frac{847735720154307}{140737488355328} & -\frac{3644178695304033}{562949953421312} & -\frac{2065817857485759}{281474976710656} \\
\frac{256866650877453}{562949953421312} & -\frac{2501054473773415}{1125899906842624} & -\frac{1008218866207467}{281474976710656} & -\frac{6366578400890641}{2251799813685248} \\
 \frac{4819157742253121}{2251799813685248} & \frac{8130105639853323}{18014398509481984} & \frac{2019419188776927}{1125899906842624} & \frac{2348470109478957}{281474976710656} \\
 \frac{451884763914337}{281474976710656} & \frac{7742823693376695}{18014398509481984} & \frac{2908076698229235}{2251799813685248} & -\frac{6004115926843261}{1125899906842624} \\
0 & -\frac{3545908575737889}{288230376151711744} & \frac{3082103085688847}{18014398509481984} & \frac{3631847684938235}{1125899906842624} \\
 -\frac{6392823462111415}{72057594037927936} & 0 & \frac{226196870745667}{9007199254740992} & -\frac{5233938836038739}{2251799813685248} \\
 -\frac{3454102047914875}{9007199254740992} & -\frac{2760923356849777}{36028797018963968} & 0 & -\frac{664739622472857}{2251799813685248} \\
 -\frac{1957421956488181}{4503599627370496} & -\frac{7174137438907087}{4503599627370496} & -\frac{4503200630060789}{36028797018963968} & 0 \\
\end{bmatrix} .$$

\end{document}